\documentclass[reqno]{amsart}
\usepackage{color}
\usepackage{hyperref}
\usepackage{graphicx}
\usepackage[all]{xy}
\newtheorem{thm}{Theorem}[section]
 
 \newtheorem{lem}[thm]{Lemma}
 \newtheorem{prop}[thm]{Proposition}

\theoremstyle{definition}
 \newtheorem{defn}[thm]{Definition}
 \theoremstyle{remark}
 \newtheorem{rem}[thm]{Remark}
  \theoremstyle{example}
  \newtheorem{exm}[thm]{Example}

 \newcommand{\Real}{\mathbb{R}}
 \newcommand{\N}{\mathbb{N}}

\begin{document}
\renewcommand{\thefootnote}{\arabic{footnote}}


\title[\it{Certain properties of Generalization  of $L^p-$SPACES FOR $0 < p < 1$ } ]{}
\begin{center}{\large{CERTAIN PROPERTIES OF GENERALIZATION OF $L^p-$SPACES FOR $0 < p < 1$}}
\end{center}

\author[ R. E\MakeLowercase{larabi}  M. E\MakeLowercase{l-arabi}   \MakeLowercase{and} M. R\MakeLowercase{houdaf} ]{R\MakeLowercase{abab}
  E\MakeLowercase{larabi}$^{  1   * },$ \ M\MakeLowercase{ouhssine}
 E\MakeLowercase{l-}A\MakeLowercase{rabi}$^2$ \    \MakeLowercase{and}  \ M\MakeLowercase{ohamed} R\MakeLowercase{houdaf}$^3$ 
}
\maketitle
\vspace*{-5mm}\begin{center}\footnotesize

$^{1,3}$Laboratoire de Math\'ematiques et leurs Applications, \'equipe:
EDP et Calcul Scientifique, Universit\'e Moulay Ismail
Facult\'e des Sciences, Mekn\`es, Morocco \\ [5mm]
$^{2}$Laboratoire
(MSISI),
 \'equipe:  Analyse fonctionnelle, Th\'eorie spectrale, Th\'eorie des codes et Applications, Universit\'e Moulay Ismail
Facult\'e des Sciences et T\'echniques, Errachidia, Morocco \\ [5mm]
\end{center}
\vspace*{-5mm}\begin{center}\footnotesize

$^{1}$\texttt{ rababelarabi162@gmail.com}\\
$^{2}$\,\,\,\texttt {elarabimouh@gmail.com}\,\,\,\,\,\,\,\,\,\\
$^{3}$\texttt{ rhoudafmohamed@gmail.com}\\
$^*$ Corresponding author.\\
\end{center}

\begin{abstract}
		This paper introduces the notion of $N^*-$function and gives a generalization of $L^p,$ for  $0<p<1$ denoted by $L_\Phi$ where $\Phi$ is an $N^*-$function. As well as, this paper examines some properties regarding to this generalized spaces and its linear forms, including some analogies and common features  to some other well known spaces.  As well as, we prove this space is a
			 quasi-normed space but it is not normed space. 
\end{abstract}
\section{\textbf{Introduction}}
Function spaces, in particular $L^p(X)$ spaces, play a central role in many questions in analysis. It is a space of measurable functions $f$ on $X,$ for $p<\infty$ have absolute values $p-$th power integrable with respect to Lebesgue measure for which 
 $\int_X |f|^p<\infty.$
For $p\geq 1,$ $L^p(X)$ is a Banach space with  the following norm:
 $ \|f\|_{L^p}= \big(\int_{X}|f|^p dx\big)^{1/p}.$
  Observe that since the triangle inequality $\|f_1+f_2\|_{L^p}\leq\|f_1\|_{L^p}+\|f_2\|_{L^p}$ fails in general when $0 <p< 1,$ $\|.\|_{L^p}$ is not a norm on $L^p$ for this range of $p,$ hence it is not a Banach space.
  This suggests that while many theorems on Banach spaces which  can be applied to the spaces $L^p$ with $p\geq 1$  may fail to hold in those spaces with $0<p<1.$ 
 
 In this work we introduce a new notion of functions called $N^*-$function and  generalization  of $ L^p$ spaces for $0<p<1$ denoted by $L_\Phi$ where $\Phi $  is an $N^*-$function.   We  concentrate on the basic structural facts about  $ L_\Phi,$ we present the linear form of $L_\Phi$ and the dual space is developed. 
 The most  theorem about Banach spaces is the Hahn-Banach theorem, which
  links the original Banach space with its dual space and  has no obvious
 extension because the dual space is zero and $L^p$ with $0<p<1$ as a  model.
  We'll see how to make the proof work for our spaces $L_\Phi.$ \\
  

The paper is organized as follows: In order not to disturb our discussions later on
we use Section 2 to present some preliminaries, including some concavity results. In Section 3, we introduce the notion of $N^*-$function and its properties 
 while the discussion of the  $L_\Phi$ spaces   presented in Section 4 and 5 . The   $L_\Phi$ spaces with dual space zero is
  developed in Section 6.
	
\section{\textbf{Preliminaries }}
	
In this section,
 we recall some basic tools that are  important in  our main results.\\
 Throughout this paper our vector spaces are real vector spaces.	

\begin{defn}

Let $X$ be a vector space. A norm on $X$ is an assignment of a norm $\|x\| \in \mathbb{R}$ to every  "point" $x$ in $X$
(that is $\|\|: X \mapsto \mathbb{R}$) satisfying the following:
\begin{itemize}
\item (Separation) For all $x \in X,$ \ $ \|x\| \geq 0$ and $\|x\| = 0$ if and only if $x =0$ 
\item (The triangle inequality) For all $x,y \in  X,$ 
$\|x + y\| \leq\| x\| + \|y\|$
\item (Homogeneity) For all $x, y \in  X,$
$\|\alpha x\| = |\alpha|\|x\|.$
\end{itemize}
Given a norm on a vector space, we get a metric by $d(x,y) = \|x -y\|.$
\end{defn}

%
\begin{defn}
A Banach space is a complete  vector space $X$ equipped with a norm $\|.\|,$
 defined by a  metric  $d(x,y) = \|x-y\|.$  
\end{defn}

%
%
\begin{defn}
A topological vector space is called locally convex if the convex open sets
are a base for the topology: given any open set $U$ around a point, there is a convex open
set $C$ containing that point such that $C \subset U.$
\end{defn}
\begin{exm}
Any Banach space is locally convex, since all open balls are convex. This follows from the definition of a norm.\\
The spaces $L^p$ for $p\geq 1$ are Banach spaces so, they are locally
convex, but in general it is not locally convex as an examples 2.19 and 2.20 in \cite{K} that, for  $0 < p < 1,$ $L^p[0, 1]$ is not locally convex.

\end{exm}
\begin{defn}
A set $A\in \mathcal{M}$ is called atom (in symbols $A\in \Lambda$) if, $\mu(A) > 0$ and for any measurable subset $B$ of $A$ with $\mu(B)<\mu(A)$ we have $\mu(B) = 0.$\\
A measure is called non-atomic if for any measurable set $A$ with $\mu(A) > 0,$ there is a measurable subset $B$ of $A$ such that
$$0 < \mu(B) < \mu(A) .$$
That is, $A$ is not an atom.\\
A measure space $(X,\mathcal{M}, \mu )$ is called an atomic if there is a countable partition $\{A_i\}$ of $X$ formed by atoms.\\
A set $A\in \mathcal{M}$ is called non-atomic (or atomless)  if neither $A$ nor any of it measurable subsets is the atom. A measure space $(X,\mathcal{M}, \mu )$ is called a non-atomic ( or atomless) if for any measurable set $A$ with $\mu(A)>0$ is an atomless. Hence $(X,\mathcal{M}, \mu )$ is non-atomic if and only if $\Lambda=\emptyset.$

There are several obvious properties of the sets of $\Lambda.$ 

\begin{enumerate}
\item If $A$ has only one element and $\mu(A)>0,$ then $A\in \Lambda.$
\item If $A\in \Lambda$ and $f$ is function measurable, then $f$ is equal to a constant $\overline{f}$ almost everywhere on $A.$
\item If $A_1\in \Lambda$ and $A_2\in \mathcal{M},$ either $\mu(A_1\cap A_2)=0$ or $\mu(A_1\cap A_2)=\mu(A_1).$ 
\item $\Lambda=\emptyset$ if and only if for every $A\in \mathcal{M},$ with $\mu(A)>0,$ there exists a sequence of measurable sets $\{B_n  \}_{n\geq 1}$ such that $B_{n+1}\subset B_n\subset A,$ $\mu(B_n)>0$ and $\lim\limits_{n\rightarrow +\infty}\mu(B_n)=0.$
\item If $A\in \Lambda  $ and $\mu(A')=\mu(A")=0,$  then $A_0=A+A'-A"\in \Lambda.$\\ 
\end{enumerate}

\end{defn}

Let's recall some definitions and properties  which will be used in the sequel of this paper.  

	\begin{defn}[\textbf{$N-$function\footnote {{\it }For more details about N-functions see  \cite{Kra}.}}]
	Let $M\colon\Real\to  \Real^{+}$ be an $N$-function, i.e., $M $
	is a convex function, with $M(t)>0$ for $t\neq0,$   
	  $$\frac{M(t)}{t}\to 0
	\ \mbox{as} \ t\to 0$$
	 and $$\frac{M(t)}{t}\to  \infty \ \mbox{ as} \ t\to  \infty .$$
	\end{defn}
	Equivalently, $M$ admits the following representation: $$M(t)=\int_0^{|t|}
	m(s)\,{\rm d}s $$ where $m\colon \Real^{+}\to \Real^{+}$ is a non-decreasing
	and right continuous function, with $m(0)=0$, $m(t)>0$ for $t>0$, and $m(t)\to
	\infty $ as $t\to \infty$.\\
	 The $N$-function $\bar{M}$ conjugate
	to $M$ is defined by $$\bar{M}(t)=\int_0^{|t|}\bar{m}(s)\,{\rm d}s ,$$
	where $\bar{m}\ :\ \Real^{+}\to  \Real^{+}$ is given by
	$\bar{m}(t)=\sup \{s\,/\, m(s)\leq t\}$.
	
	\begin{defn}[\textbf{Orlicz spaces}](see \cite{A,Kra})
	 Let $X$ be an open subset of $\Real^d$, $d\in\N$. The Orlicz class $
	{\mathcal{L}_M}(X)$ (resp.~the Orlicz space $L_M(X)$) is
	defined as the set of (equivalence classes of) real-valued
	measurable functions $u$ on $X $ such that:
	\begin{equation}\label{10}
	\displaystyle \int_X M(u(x))\,{\rm d}x<+\infty \quad {\mbox{(resp.\ }}
	\int_X M\Big(\frac{u(x)}\lambda \Big)\,{\rm d}x<+\infty {\mbox{ for some
	$\lambda>0$}}).
	\end{equation}
	\end{defn}
	Notice that $L_M(X)$ is a Banach space under the so-called Luxemburg
	 norm, namely
	\begin{equation}\label{11}
	\| u\| _{M}=\inf \Big\{ \lambda >0\,/\, \int_X
	M\Big(\frac{u(x)}{\lambda}\Big)\,{\rm d}x\leq 1\Big\}
	\end{equation}
	and ${\mathcal{L}_M}(X)$ is a convex subset of $L_M(X)$.
\begin{defn}[\textbf{Concave function}]
A real-valued function $\Phi(x)$ of the real
	variable $x$ is said to be concave if the inequality
	\begin{equation}\label{concave}
	\Phi(\alpha x + ( 1 -\alpha )y)\geq \alpha \Phi(x)+(1-\alpha )\Phi(y)
	\end{equation}
is satisfied for all values of $x, y$ and  $0\leq \alpha \leq 1 .$
\end{defn}
Inequality (\ref{concave}) admits still another generalization:
 
 \begin{equation}\label{Gconcave }
 	\displaystyle\Phi\big(\frac{x_1+x_2+...+x_n}{n}\big)\geq\frac{1}{n}[  \Phi(x_1)+\Phi(x_2)+...+\Phi(x_n)]
 	\end{equation}
for arbitrary $x_l, x_2, ... , x_n.$ By successive application of (\ref{concave}). 

\begin{lem}\label{lem1}
A continuous concave function $\Phi(x)$ has, at every
point, a right derivative $p_+ (x)$ and a left derivative $p_- (x)$ such that
\begin{equation}\label{p}
p_-(x)\geq p_+(x).
\end{equation}
\end{lem}

\begin{lem}\label{1.3}
A concave function $\Phi$ is absolutely continuous
and satisfies the Lipschitz condition in every finite interval.
\end{lem}
The above Lemmas \ref{lem1} and  \ref{1.3}   are proven by the same way as that indicated respectively in the proof of   lemma 1.1 and   lemma 1.3 in \cite{Kra}. 
%
%
%

\begin{thm}\label{thm}
Every concave function $\Phi$ which satisfies the
condition $\Phi(a)=0$ can be represented in the form
\begin{equation}\label{form}
\Phi(x)=\int_{a}^{x}p(t)dt,
\end{equation} 
where $p$ is a decreasing right-continuous function.
\end{thm}
\begin{proof}
We note first of all that the function $\Phi$ has a
derivative almost everywhere. In fact ,
we have in virtue of (\ref{p}) 
and 
\begin{equation}
p_+(x_1)\geq p_-(x_2).
\end{equation}
%
%
so, we have that
\begin{equation}\label{1.10}
p_-(x_2)\leq p_+(x_1)\leq p_-(x_1)
\end{equation}
for $x_2>x_1.$ Since the function $p_-$ is monotonic, it is continuous almost everywhere. Let $x_1$ be a point of continuity of the function
$p_-.$  Passing to the limit in (\ref{1.10}) as $x_2\rightarrow x_1,$ we obtain that
$p_-(x_1)\leq p_+(x_1) \leq p_-(x_1) ,$  i.e. $p_-(x_1) = p_+(x_1) .$\\
Similarly, we have that $\Phi'(x) = p(x) = p_+(x)$ almost everywhere.\\ 
Since the function $\Phi$ is absolutely continuous (in virtue of
Lemma \ref{1.3}), it is the indefinite integral of its derivative.

\end{proof}

\section{ \textbf{$N^*$-function.}}

A function $\Phi$ is called an
$N^*-$function if it admits of the representation
\begin{equation}\label{N*-function}
\displaystyle \Phi(x)=\int_{0}^{|x|}p(t)dt<+\infty;
\end{equation}
where the function $p(t)$ is right-continuous for $t\geq 0,$ positive
for $t > 0$ and decreasing which satisfies the conditions

\begin{equation}\label{limit_p}
 \lim\limits_{t\rightarrow 0^+}p(t)=p(0^+)=+\infty, \ \lim\limits_{t\rightarrow +\infty}p(t) =p(+\infty)=0.
\end{equation}
For example, the functions   
$\displaystyle \Phi_1(x)=\displaystyle e^{\frac{\ln(\alpha|x| )}{\alpha}}\ (\alpha>1),\ \Phi_2(x)=\sqrt{\ln(|x|+1)}$
are $N^*-$functions. For the first of these, $\displaystyle p_1(t) = \Phi'_1(t) = \frac{1}{\alpha t}e^{\frac{\ln(\alpha t)}{\alpha}}$ and,
for the second, $\displaystyle p_2(t) = \Phi'_2(t) = \frac{1}{2(t+1)\sqrt{\ln(t+1)}}.$ 

\subsection{Properties of $N^*-$functions.}

It follows from representation
(\ref{form}) that every $N^*-$function is even, continuous, assumes the value
zero at the origin, and non-decreases for positive values of the argument.
$N^*-$functions are concave. In fact , if $0< x_1 < x_2,$ then, in virtue
of the monotonicity of the function $p,$ we have that
\begin{eqnarray*}
\Phi(\frac{x_1+x_2}{2}) =\int_{0}^{\frac{x_1+x_2}{2}}p(t)dt&=&\int_{0}^{\frac{x_1}{2}}p(t)dt+\int_{\frac{x_1}{2}}^{\frac{x_1+x_2}{2}}p(t)dt\\
&=& \int_{0}^{\frac{x_1}{2}}p(t)dt+\int_{0}^{\frac{x_2}{2}}p(t+\frac{x_1}{2})dt\\
&\geq& \frac{1}{2}\int_{0}^{\frac{x_1}{2}}p(t)dt+\frac{1}{2}\int_{0}^{\frac{x_2}{2}}p(t)dt\\
&\geq& \frac{1}{2}[\Phi(x_1)+\Phi(x_2)].
\end{eqnarray*}
In the case of arbitrary $x_1, x_2,$ we have that

\begin{eqnarray*}
\Phi(\frac{x_1+x_2}{2}) =\Phi(\frac{|x_1+x_2|}{2})&\geq& \Phi(\frac{|x_1|+|x_2|}{2})\\
&\geq& \frac{1}{2}[\Phi(x_1)+\Phi(x_2)].  
\end{eqnarray*}
Setting $x_2 = 0$ in (\ref{concave}) , we obtain that
\begin{equation}\label{1.13}
\Phi(\alpha x_1)\geq \alpha \Phi(x_1) \ \ (0\leq \alpha \leq 1).
\end{equation}
The first of conditions (\ref{limit_p}) signifies that
\begin{equation}\label{1.14}
\lim\limits_{x\rightarrow 0^+}\frac{\Phi(x)}{x}=+\infty.
\end{equation}
It follows from the second condition in (\ref{limit_p}) that
\begin{equation}\label{1.15}
\lim\limits_{x\rightarrow +\infty}\frac{\Phi(x)}{x}=0.
\end{equation}
We denote by $\Phi^{-1} \  (0\leq x< \infty)$ the inverse function to the
$N^*-$function $\Phi$ considered for non-negative values of the argument. This function is convexe since, in virtue of inequality (\ref{concave}) ,
we have that $\Phi^{-1}[\alpha x_1 + (1 - \alpha) x_2]\leq  \alpha\Phi^{-1}( x_1) + (1 - \alpha)\Phi^{-1} (x_2)$
for $x_1 , x_2\geq 0.$

The monotonicity of the right derivative $p$ of the $N^*-$function
$\Phi$ implies the inequality

\begin{eqnarray}\label{1.16}
\Phi(x+y)=\int_{0}^{|x+y|}p(t)dt&\leq &\int_{0}^{|x|+|y|}p(t)dt\nonumber\\
&=&\int_{0}^{|x|}p(t)dt+\int_{|x|}^{|x|+|y|}p(t)dt\nonumber\\
&\leq& \int_{0}^{|x|}p(t)dt+\int_{0}^{|y|}p(t+|x|)dt\nonumber\\
&\leq& \int_{0}^{|x|}p(t)dt+\int_{0}^{|y|}p(t)dt.
\end{eqnarray}

\subsection{Second definition of an $N^*-$function}

 It is sometimes expedient
to use the following definition: a continuous concave function $\Phi$ is called an $N^*-$function if it is even and satisfies conditions (\ref{1.14}) and (\ref{1.15}). We shall show that this definition is equivalent to the  fact that it follows from the second definition of an $N^*-$function that it is possible to represent it in the form (\ref{N*-function}).

In virtue of (\ref{1.14}), $\Phi(0) =0.$ Therefore, in virtue of the evenness
of the function $\Phi$ and Theorem \ref{thm}, it can be represented in
the form
\begin{equation*}
\Phi(x)=\int_{0}^{|x|}p(t)dt,
\end{equation*}
where the function $\Phi$ is non-decreasing for $x > 0$ and rightcontinuous
(i.e. the right derivative of the function $\Phi$). For $\displaystyle x > 0, \ \frac{\Phi(x)}{x}=\frac{1}{x}\int_{0}^{x}p(t)dt\geq \frac{1}{x}\int_{\frac{x}{2}}^{x}p(t)dt\geq \frac{p(x)}{2}>0,$ and in virtue of (\ref{1.15}), $\lim\limits_{x\rightarrow +\infty}p(x)=0.$ On the other hand, $\displaystyle p(0^+)=\lim\limits_{x\rightarrow 0^+}p(x)=\lim\limits_{x\rightarrow 0^+}\frac{\Phi(x)}{x}=+\infty.$

\begin{prop}\label{phi-1}
A function $\Phi$ is an $N^*-$function if and only if $\Phi^{-1}:\mathbb{R}^+\rightarrow\mathbb{R}^+$ is an $N-$function.
\end{prop}
\begin{proof}
\begin{enumerate}
\item $\Phi^{-1}(0)=0\Leftrightarrow \Phi(0)=0$   (from the fact that $\Phi$ is bijective).
\item 
\begin{eqnarray*}
\displaystyle\lim\limits_{x\rightarrow 0^+}\frac{\Phi^{-1}(x)}{x}&=&\lim\limits_{x\rightarrow 0^+}\frac{\Phi^{-1}(x)}{\Phi(\Phi^{-1}(x))}\\
&=&\displaystyle\lim\limits_{x\rightarrow 0^+}\displaystyle\frac{1}{\frac{\Phi(\Phi^{-1}(x))}{\Phi^{-1}(x)}}\\
&=& \displaystyle\lim\limits_{X\rightarrow 0^+}\frac{1}{\frac{\Phi(X)}{X}} =0.
\end{eqnarray*}
\item
 \begin{eqnarray*}
\displaystyle\lim\limits_{x\rightarrow +\infty}\frac{\Phi^{-1}(x)}{x}&=&\lim\limits_{x\rightarrow +\infty}\frac{1}{\frac{\Phi(\Phi^{-1}(x))}{\Phi^{-1}(x)}}\\
&=&\displaystyle\lim\limits_{x\rightarrow +\infty}\displaystyle\frac{1}{\frac{\Phi(\Phi^{-1}(x))}{\Phi^{-1}(x)}}\\
&=& \displaystyle\lim\limits_{X\rightarrow +\infty}\frac{1}{\frac{\Phi(X)}{X}} =+\infty.
\end{eqnarray*}
The inverse implication proceed as the same way.
\end{enumerate}
\end{proof}

\begin{rem}
we finish this section  by the complementary  $N^*-$function defined as follows:
$$\widehat{\Phi}=\big(\overline{\Phi^{-1}})^{-1}$$ 
deduced from the young's inequality,   which is also  an $N^*-$function.

\end{rem}
\textbf{Example.} As an example of $\widehat{\Phi},$ let $X=[0,a]$ with $a\geq1$ and $\Phi(t)=\displaystyle\frac{t^p}{p^p}$ where $0<p\leq 1.$
For any $t\geq 0$ we have 
$$\Phi^{-1}(t)=\frac{t^{1/p}}{1/p} \ \mbox{and} \ \overline{\Phi^{-1}}(t)=\frac{t^{\frac{1}{1-p}}}{\frac{1}{1-p}}.$$
So that 
$$\widehat{\Phi}(t)=\big(\overline{\Phi^{-1}})^{-1}(t)=\frac{t^{1-p}}{(1-p)^{1-p}}.$$
\section{\textbf{The $L_\Phi$ spaces}}
In this section, we present a new definition of a generalized spaces denoted by $L_{\Phi}$  which is a generalisation of $L^p$ spaces with $0<p<1,$ for any measure space and then extend some results in this generalized framework.
 
 \subsection{Definition and Elementary Properties of $L_\Phi$ Spaces}
 

\begin{defn}
Let $(X,\mathcal{M},\mu)$ be a measure space and $\Phi$ be an $N^*-$function, we define the space 
\begin{equation*}
L_\Phi(X)\stackrel{def}{=}:\{  f:X\rightarrow \mathbb{R} \ \mbox{measurable and } \ \int_{X}\Phi(f )d\mu<\infty \},
\end{equation*}
with functions that are equal almost everywhere being identified with one another.
\end{defn}

\begin{rem}

The  $L_\Phi(X)$ is a vector metric spaces and the metric used on it is defined  as follows :
\begin{equation*}
d_\Phi(f,g)=\int_{X}\Phi(f-g)d\mu.
\end{equation*}

\begin{itemize}
\item $d_\Phi(f,g)=0 \Leftrightarrow \int_{X}\Phi(f-g)d\mu=0 \Leftrightarrow f=g.$ 
\item If $f, g, h \in L_\Phi(X),$ then 
\begin{eqnarray*}
d_\Phi(f,g)=\int_{X}\Phi(f-g)d\mu&=& \int_{X}\Phi(f-h+h-g)d\mu\\
&\leq& \int_{X}\Phi(f-h)d\mu+\int_{X}\Phi(h-g)d\mu\\
&\leq& d_\Phi(f,h)+d_\Phi(h,g).
\end{eqnarray*}
\end{itemize}
\end{rem}

The notion of these spaces extends the usual notion of $L^p$ with $(0<p<1).$ The function  $s \mapsto s^p$ entering the definition of $L^p$ is replaces by a more general concave function $s\mapsto\Phi(s)$ which is the $N^*-$function.\\
Let's give some analogies between the Orlicz Spaces 
$L_M$
	 where $M$ is an $N-$function and our Spaces $L_\Phi,$ where $\Phi$ is an $N^*-$function.

Some properties of the Orlicz spaces $L_M$ are:


\begin{itemize}
\item[$P_1)$] $(L_M,||.||_M)$ is  complete, and thus is a Banach space for the Luxemburg norm $||.||_M.$
\item[$P_2)$] (H\"{o}lder's inequality)
\[
\int_{X}|u(x)v(x)|\,{\rm d}x \le ||u||_{M}||v||_{\bar{M}}
\mbox{ for all $u\in L_{M}(X)$ and $v\in L_{\bar{M}}(X)$,}
\]
where $M$ is an $N-$function and $\bar{M}$ its complementary function.
\item[$P_3)$](Jensen's inequality)  
$$M\big( \frac{1}{\mu(X)}\int_Xf(t) d\mu\big)\leq  \frac{1}{\mu(X)} \int_X M(f(t)) d\mu$$
for all $f\in L_M(X).$
\item[$P_4)$] In particular, if $X$ has finite measure, both  H\"older's inequality  and Jenson's inequality yield
 the continuous inclusion $L_{M}(X)\subset L^1(X)$.\\ 
For more properties about Orlicz spaces see \cite{A,Kra,M,R}.
\end{itemize}

Let's look at the $L_\Phi$ spaces and develop more its properties:

\begin{thm}
For $\Phi$ is $N^*-$function,   $(L_\Phi,d_\Phi)$ is a complete space.
\end{thm}
\begin{proof}
Let $(f_n)$ be a cauchy sequence of $L_\Phi.$
We select a representatives and identify them at $f_n.$ We can build a subsequence $(f_{n_i})$ verifies 
$$ d_\Phi(f_{n_{i+1}},f_{n_i})\leq 2^{-i}$$ 
then we pose 
$$g_k=\sum_{i=1}^{k}|f_{n_{i+1}}-f_{n_i}|$$
and $$g=\sum_{i=1}^{+\infty}|f_{n_{i+1}}-f_{n_i}|$$
these are functions with values in  $[0,+\infty],$ triangular inequality shows that:
$$d_\Phi(g_k,0)\leq 1.$$
 Fatou's lemma or the monotone convergence theorem then gives 
$$d_\Phi(g,0)\leq 1$$
In particular this shows that these functions are finite  almost everywhere and therefore for almost any $x,$ the series $\displaystyle\sum f_{n_{i+1}}(x)-f_{n_i}(x)$ 
converge absolutely. As $ \mathbb{R}$ is complete. This series converges and since it is a telescopic series, it means that $f_{n_i}(x)$ converges to the sum of the series noted $f(x). $\\ $$\displaystyle \sum_{i=1}^{k}f_{n_{i+1}}(x)-f_{n_i}(x)=f_{n_k}(x)-f_{n_{k-1}}(x)+f_{n_{k-1}}(x)-f_{n_{k-3}}(x)-...-f_{n_{1}}(x))$$
 This defines almost everywhere, then one completes by $0. $ It remains to be shown that $f$ is in $L_\Phi$ and $(f_n)$ converges to $f$ for distance $d_\Phi$ of $L_\Phi.$\\
Let’s go back to the fact that $(f_n)$ is  Cauchy sequence: 
$$\forall \varepsilon>0, \exists N, \forall n,m\geq N, \ \int_{X}\Phi(f_n-f_m)<\Phi(\varepsilon).$$
By taking $n=n_i,$ we have by  Fatou's lemma: $$\int_X\Phi(f-f_m)d\mu\leq \lim\limits_{}\inf \int_X \Phi(f_{n_i}-f_m)d\mu<\Phi(\varepsilon)$$
and 
$$\displaystyle\int_X\Phi(f)d\mu= \int_X\Phi(f-f_m+f_m)d\mu\leq \int_X\Phi(f-f_m)d\mu+\int_X\Phi(f_m)d\mu<\infty$$
then $f\in L_\Phi, \ \mbox{and}\ f_n\rightarrow f \ \mbox{in} \ L_\Phi.$\\
\end{proof}

\begin{thm}\label{ThmY}
For $f\in L_\Phi$ and $g\in L_{\widehat{\Phi}}$
\begin{equation}\label{notHolder}
\int_{X}\Phi(f) \widehat{\Phi}(g)\,{\rm d}x \le \int_{X}|f|dx+\int_{X}|g|dx
\end{equation}
where  $\widehat{\Phi}$ is a complementary  $N^*-$function.
\end{thm}
\begin{proof}
it sufficient to apply the following  Young's inequality:
$$|fg|\leq \Phi^{-1}(f)+\overline{\Phi^{-1}}(g)$$
where  $\Phi^{-1}$ is an $N-$function owing to Proposition \ref{phi-1}. We get (\ref{notHolder}) with  $\widehat{\Phi}=[\overline{\Phi^{-1}}]^{-1}.$ Moreover, it is clear to show that $\widehat{\widehat{\Phi} } =\Phi.$
\end{proof}

If $X$ has finite measure we have the following results.

\begin{prop}\label{Ojens}
For $f\in L_\Phi(X),$ 
 then 
\begin{equation*}
\Phi(\frac{1}{\mu(X)}\int_{X}|f(t)|d\mu)\geq \frac{1}{\mu(X)}\int_{X}\Phi(f)d\mu.
\end{equation*}
\end{prop}
\begin{proof}
The $N^*-$function $\Phi$ is  convex, then it can be written as the upper envelope of refined functions
\begin{equation*}
\Phi(x)= \inf_{n\geq0}(a_nx+b_n),
\end{equation*}
where  $a_n,b_n \in \mathbb{R}.$ Which gives
\begin{eqnarray*}
\Phi(\frac{1}{\mu(X)}\int_{X}|f(t)|d\mu)&=&\inf_{n\geq0}\big[\frac{a_n}{\mu(X)}\int_{X}|f(t)|d\mu+b_n\big]\\
&\geq&\frac{1}{\mu(X)}\int_{X}\inf_{n\geq0}(a_n|f(t)|+b_n)d\mu\\
&\geq&\frac{1}{\mu(X)}\int_{X}\Phi(f)d\mu. \\
\end{eqnarray*}
\end{proof}
From Theorem \ref{ThmY} and Proposition \ref{Ojens}  yield
 the continuous inclusion $L^1(\Omega)\subset L_{\Phi}(\Omega) $.
\begin{thm}
Let $f_n,f\in L^1(X),$ if
 $\parallel f_n-f\parallel_1\rightarrow0,$ then $d_\Phi(f_n,f)\rightarrow0.$
\end{thm}
\begin{proof}
We use  the fact that  $L^1\subset L_\Phi(X),$ as $f_n,f\in L^1(X),$ then $f_n-f\in L_\Phi(X).$ So 
   $$d(f_n,f)=\Phi(|f_n-f|)\leq \mu(X)\Phi(\frac{1}{\mu(X)}\int_X|f_n-f|d\mu)$$
   as $\parallel f_n-f\parallel_1\rightarrow0,$ and $\Phi$ is function continuous, then $d(f_n,f)\rightarrow0.$
\end{proof}

\begin{thm}
The elements of $L^1$ form a dense in $L_\Phi(X).$
\end{thm}
\begin{proof}
let $f\in L_\Phi\setminus L^1$ and $(f_n)_{n=1}^\infty$ be the simple function is the approximation of $f.$ 
 Since $|f_n|\leq |f|$ for all $n,$ and 
 $$\Phi(|f-f_n|)\leq \Phi(|f|)+\Phi(|f_n|)\leq \varepsilon\Phi(|f|)\in L^1.$$
Therefore, by the dominated convergence theorem
$$\lim\limits_{n\rightarrow+\infty}d(f_n,f)=\lim\limits_{n\rightarrow+\infty}\int_{X}\Phi(|f_n-f|)d\mu=\int_{X}\lim\limits_{n\rightarrow+\infty}\Phi(|f_n-f|)d\mu=0.$$ 
\end{proof}
We can  conclude from this analogies that our spaces $L_\Phi$ is including $L^1, L^p$
 for $0<p<1$ and the Orlicz spaces $L_M$ where $M$ is $N-$function.\\

It is known that if $M$ is $N-$function, then for every $\alpha,$ we have $\alpha<M^{-1}(\alpha)+(M^*)^{-1}(\alpha)\leq 2\alpha.$ Since $\Phi$ is $N^*-$function, then $\Phi^{-1}$ is  $N-$function, therefore:
\begin{equation}\label{result}
\alpha<\Phi(\alpha)+\widehat{\Phi}(\alpha) \leq 2 \alpha
\end{equation}
  for $\alpha>0.$ By this property we can get the following proposition: 

\begin{prop}
We have $L_{\Phi}\cap L_{\widehat{\Phi}}=L^1.$
\end{prop}

\begin{proof}
Owing to the  inequality (\ref{result}), we have for all $f\in L^1$ 
\begin{equation*}
\int_X \Phi(|f|)d\mu +\int_X {\widehat{\Phi}}(|f|)d\mu\leq 2 \int_X \Phi(|f|)d\mu <+\infty,
\end{equation*} 
then 
$$\int_X \Phi(|f|)d\mu<+\infty \ \mbox{and}\ \int_X \widehat{\Phi}(|f|)d\mu<+\infty . $$
Hence $$f\in L_{\Phi}\cap L_{\widehat{\Phi}}. $$
Let now $ f\in L_{\Phi}\cap L_{\widehat{\Phi}};$ then $\int_X \Phi(|f|)d\mu<+\infty $ and $\int_X {\widehat{\Phi}}(|f|)d\mu<+\infty,$ so 
$$\int_X |f|d\mu<\int_X \Phi(|f|)d\mu+\int_X \widehat{\Phi}(|f|)d\mu<+\infty,$$
therefore, $f\in L^1,$ so
 that $L_{\Phi}\cap L_{\widehat{\Phi}}\subseteq L^1.$ Then we get the result.
\end{proof}



   \subsection{\textbf{$L_\Phi$ spaces with $\Phi^{-1}\in \Delta_2$}}
   
 \vspace*{0.5cm}Let's recall first the  definition of  $\Delta_2$-condition:
 \begin{defn}
 An $N$-function $M$ is said to satisfy the $\Delta _2$-condition if,
 for some $k>2$,
 \begin{equation}\label{7}
 M(2t)\leq k\,M(t)\quad {\mbox{for all }}t\geq 0.
 \end{equation}
We can denote by $M\in \Delta_2$ when $M$ is satisfying the $\Delta_2$-condition.

 \end{defn}
 
 \begin{rem}
 Let $\Phi^{-1}\in \Delta_2,$  then there exists $k_0>2$ such that, $\Phi^{-1}(2x)\leq k_0\Phi^{-1}(x); \ \forall x>0,$  if and only if there exists $ k_0>2 $ such that
 \begin{equation}\label{Ndelta2}
    2\Phi(x)\leq \Phi(k_0x) \ \mbox{for all}\ x>0.
 \end{equation} 
   We take $\psi_x(t)=\Phi(tx)-2\Phi(x), \ \forall t\geq 2 \ \mbox{and}\ \forall x>0. $
   Then 
   \begin{itemize}
   \item[i)] $\psi_x$ is continuous on $[2,+\infty[$ for each $x>0.$
    \item[ii)] $\psi_x(k_0)=\Phi(k_0x)-2\Phi(x)\geq 0 \ \forall x>0$ and $\psi_x(2)=\Phi(2x)-2\Phi(x)\leq 0 \ \forall x>0$
   \end{itemize}
   there is $k\in [2,k_0],$ such that $\Phi_x(k)=0 \ \forall x>0$
   equivalently to 
   \begin{equation}\label{Ndelta22}
      2\Phi(x)=\Phi(kx) \ \forall x>0.
   \end{equation} 
 \end{rem}  
 
 In the sequel we suppose that $\Phi^{-1}\in \Delta_2,$ so that $\Phi$ is satisfying the assumption (\ref{Ndelta22}).  \\
 
%

  We recall the following definition of a quasi-norm on a real vector space $X$ is a map $x\mapsto \|x\| \ (X\rightsquigarrow \mathbb{R})$ such that:
  \begin{itemize}
  \item $\|x\|>0$ if $x\neq 0,$
  \item $\|tx\|=|t|\|x\|$ for $x\in X, \ t\in \mathbb{R},$
  \item $\|x+y\|\leq k(\|x\|+\|y\|)$ for $x,y\in X,$
  \end{itemize}
where $k$ is a constant independent of $x$ and $y.$ The best such constant $k$ is called the modulus of concavity of the quasi-norm. If $k\leq 1,$ then the quasi-norm is called a norm.\\
The sets $\{x\in X: \|x\|<\varepsilon\}$ for $\varepsilon>0$ form a base of neighbourhoods of a Hausdorff vector topology on $X.$  This topology is (locally) p-convex, where $0<p\leq1$ if for some constant $A$ and any $x_1,x_2,..,x_n\in X, \ \|x_1+x_2+..+x_n\|\leq A (\|x_1\|^p+..+\|x_n\|^p)^{1/p}.$   
In this case we may endow $X$ with an equivalent quasi-norm:
$$\|x\|^* =\inf \{ (\sum_{i=1}^{n}\|x_i\|^p)^{1/p}; \ x_1+x_2+..+x_n=x \}$$ 
and then 
$$\|x\|\geq \|x\|^*\geq A^{-1}\|x\|; \ x\in X;$$ 
and $\|.\|^*$ is $p$-sub-additive, i.e.
$$\|x_1+x_2+..+x_n\|^*\leq (\sum_{i=1}^{n}\|x_i\|^{*p})^{1/p}.$$

  \begin{prop}
  The function defined by: $\displaystyle\|f\|_\Phi = \inf\{\lambda>0, \int_X \Phi(\frac{|f|}{\lambda})d\mu \leq 1  \}$ is a quasi-norm on $L_\Phi$ space. $(L_\Phi, \|\|_\Phi)$ is a quasi-normed space.
  \end{prop} 
   
 \begin{proof}
  \begin{itemize}
  \item[$1)$] If $f=0,$ then $\displaystyle\Phi(\frac{|f|}{\lambda})=0 \ \forall \lambda >0,$ so $\|f\|_\Phi=0.$  
  \item[$2)$] If $\|f\|_\Phi=0$ then $\displaystyle\int_X\Phi(\alpha|f|)d\mu \leq 1, \ \forall \alpha >0.$
  \end{itemize}  
We have $\displaystyle\int_X\Phi(|f|)d\mu=\frac{1}{2}\int_X\Phi(k|f|)d\mu=...=\frac{1}{2^n}\int_X\Phi(k^n|f|)d\mu \ \forall n\in \mathbb{N},$ by $2)$, then $\displaystyle\int_X\Phi(|f|)d\mu\leq \frac{1}{2^n} \ \forall n\in \mathbb{N}.$ While $\displaystyle\frac{1}{2^n}\rightarrow 0 \ \mbox{as} \ n\rightarrow +\infty,$ that is $\int_X\Phi(|f|)d\mu=0,$ so  $f=0.$\\

Let $\alpha$ and $f\in L_\Phi,$ then 
\begin{itemize}
\item If $\alpha=0,$ then $\|\alpha f \|_\Phi=\|0\|_\Phi=0=0\|f\|_\Phi.$

\item If $\alpha\neq0,$ then $\|\alpha f \|_\Phi=\displaystyle\inf\{ \lambda >0, \int_X \Phi(\frac{\alpha f}{\lambda})d\mu <1 \}\\
\hspace*{3.2cm}=\displaystyle\inf\{ |\alpha|\frac{\lambda}{|\alpha|}, \int_X \Phi(\frac{ f}{\frac{\lambda}{|\alpha|}})d\mu <1 \}\\
\hspace*{3.2cm}=\displaystyle|\alpha|\inf\{ \frac{\lambda}{|\alpha|}>0 , \int_X \Phi(\frac{ f}{\frac{\lambda}{|\alpha|}})d\mu <1 \}\\
\hspace*{3.2cm}=|\alpha|\|f\|_\Phi.$
\end{itemize}   
   
Let $f,g\in L_\Phi$ and $a,b \in \mathbb{R},$ such that: $a\geq \|f\|_\Phi$ and $b \geq \|g\|_\Phi,$ then:
$$\int_X \Phi(\frac{|f|}{a})d\mu<1 \ \mbox{and} \ \int_X \Phi(\frac{|g|}{b})d\mu<1,$$
since $\displaystyle\int_X \Phi\big(\frac{|f+g|}{k(a+b)}\big)d\mu=\frac{1}{2}\int_X \Phi\big(\frac{k|f+g|}{k(a+b)}\big)d\mu=\frac{1}{2}\int_X \Phi\big(\frac{|f+g|}{(a+b)}\big)d\mu\\
\hspace*{3.5cm}\leq\frac{1}{2}\big[\int_X \Phi\big(\frac{|f|}{(a+b)}\big)d\mu + \int_X \Phi\big(\frac{|g|}{(a+b)}\big)d\mu\big]\\
\hspace*{3.5cm}\leq\frac{1}{2}\big[\int_X \Phi\big(\frac{|f|}{a}\big)d\mu + \int_X \Phi\big(\frac{|g|}{b}\big)d\mu\big]
\leq 1,$\\
which means that $\|f+g\|_\Phi\leq k(a+b).$ As a result, 
$\|f+g\|\leq \big(\|f\|_\Phi+\|g\|_\Phi\big).$
 \end{proof}  
\begin{rem}
As is well known, whenever $0<p<1$ the function $\|f\|_p=\big(\int_X |f|^p d\mu\big)^{\frac{1}{p}}$ no longer satisfies the triangle inequality  $\|f_1+f_2\|_p\leq \|f_1\|_p+\|f_2\|_p$ but in general only the weaker condition $\|f_1+f_2\|_p\leq 2^{\frac{1-p}{p}}(\|f_1\|_p+\|f_2\|_p).$ Then the function $\|.\|_p$ is a quasi-norm on $L^p,$ and coincide with the quasi-norm $\|.\|_\Phi.$ Indeed, since $\Phi(t)=t^p$ is $N^*-$function, $2\Phi(t)=(2^{1/p})^p t^p=(2^{1/p}.t)^p,$ for all $t\geq 0.$ Set $k=2^{1/p}\geq 2.$
$$\int_X \big(\frac{|f|}{\|f\|_\Phi}\big)^p d\mu=\int_X \frac{|f|^p}{\|f\|_\Phi}^p d\mu$$ 
then $\big[\int_X |f|^pd\mu\big]^{1/p}\leq \|f\|_\Phi $ for all $f\in L_\Phi.$
So that $ \|f\|_p\leq \|f\|_\Phi.$\\
On the other hand, we have $\int_X \big(\frac{|f|}{\|f\|_\Phi}\big)^p d\mu=\frac{1}{\|f\|^p_p}\int_X |f|^pd\mu=1.$ So $ \|f\|_\Phi\leq \|f\|_p.$ Finally,  we conclude that $ \|f\|_p= \|f\|_\Phi.$ 

\end{rem}      
   
\begin{prop}
Let $\mu(X)<+\infty$ and $f\in L^1,$ then 
\begin{equation}
\|f\|_\Phi\leq \frac{1}{\mu(X)\Phi^{-1}(\frac{1}{\mu(X)})}\|f\|_1.
\end{equation}
\end{prop}   
  \begin{proof}
  \begin{eqnarray*}
  \int_X\Phi\big(\frac{|f|}{\|f\|_1}\mu(X)\Phi^{-1}(\frac{1}{\mu(X)})d\mu\big)&=&\int_X\big[\inf_n \{a_n \frac{|f|}{\|f\|_1}\mu(X)\Phi^{-1}(\frac{1}{\mu(X)})+b_n \}   \big]d\mu\\
  &\leq& \inf_n \{a_n \int_X \frac{|f|}{\|f\|_1}\mu(X)\Phi^{-1}(\frac{1}{\mu(X)})d\mu+b_n \mu(X) \}\\
 &\leq& \mu(X)\big[\inf_n \{a_n \int_X \frac{|f|}{\|f\|_1}\Phi^{-1}(\frac{1}{\mu(X)})d\mu+b_n  \}\big]\\
 &\leq& \mu(X)\Phi\big( \frac{1}{\|f\|_1}(\int_X |f|d\mu)\Phi^{-1}(\frac{1}{\mu(X)})\big)\\
  &\leq& \mu(X)\Phi\big(\Phi^{-1}(\frac{1}{\mu(X)})\big)\\
  &\leq& \mu(X)\frac{1}{\mu(X)}=1.
  \end{eqnarray*}
 
Then 
\begin{equation*}
\|f\|_\Phi\leq \frac{1}{\mu(X)\Phi^{-1}(\frac{1}{\mu(X)})}\|f\|_1.
\end{equation*}
   
   \end{proof} 

\begin{prop}
Let $f\in L_\Phi.$ If $\displaystyle\int_X \Phi(|f|)d\mu<c$ then $\displaystyle\|f\|_\Phi \leq k^{n_0},$ where $\displaystyle n_0=[\frac{\ln(c)}{\ln(2)}]+1.$
\end{prop}
 \begin{proof} We have 
 \begin{eqnarray*}
 \int_X \Phi(\frac{|f|}{k^{n_0}})d\mu&=&\frac{1}{2^{n_0}}\int_X \Phi(\frac{|f|k^{n_0}}{k^{n_0}})d\mu\\
 &=& \frac{1}{2^{n_0}}\int_X \Phi(|f|)d\mu \leq \frac{c}{2^{n_0}}; 
 \end{eqnarray*}
 since $\displaystyle n_0>\frac{\ln(c)}{\ln(2)},$ then $\displaystyle e^{\ln(2)n_0}>c$ so that $2^{n_0}>c$ which gives $\displaystyle\frac{c}{2^{n_0}}<1,$ we get $\displaystyle\|f\|_\Phi \leq k^{n_0}.$
 \end{proof}

 \begin{defn}
 \begin{enumerate}
 \item A sequence $\{f_n\}$ in $L_\Phi (X)$ is called a quasi-convergent to a point $f\in L_\Phi$ if and only if $\|f_n-f\|_\Phi\rightarrow 0$  and we note $\xymatrix{f_n\ar[r]^{\|.\|_\Phi}&f} \ \mbox{as}\ n\rightarrow \infty.$
 \item A sequence $\{f_n\}$ in $L_\Phi (X)$ is quasi-cauchy sequence if and only if $\|f_n-f_m\|_\Phi \rightarrow 0$ as $n,m\rightarrow \infty.$
 \end{enumerate}
 Is clear that every a quasi-convergent sequence is a quasi-cauchy sequence.
 \end{defn}  
   \begin{prop}\label{prop2}
   Let $\{f_n\}_{n\in \mathbb{N}}$ be a sequence in $L_\Phi(X)$ and $f\in L_\Phi(X).$ Then 
   $\xymatrix{f_n\ar[r]^{\|.\|_\Phi}&f}$ if and only if $\xymatrix{f_n\ar[r]^{d_\Phi}&f.}$
     \end{prop}
 \begin{proof}
 $\Rightarrow]$ Assume that $\xymatrix{f_n\ar[r]^{\|.\|_\Phi}&f.}$ Then, for all $i\in \mathbb{N},$ there exists $N_i,$ for all $n>N_i,$  $\|f_n-f\|_\Phi\leq \frac{1}{k^i}.$ This implies that : 
 \begin{equation*}
 2^i \int_X \Phi(|f_n-f|)d\mu=\int_X \Phi(k^n|f_n-f|)d\mu=\int_X \Phi(\frac{|f_n-f|}{k})d\mu<1. 
 \end{equation*}
 
 Hence, 
 \begin{equation*}
 d_\Phi(f_n,f)= \int_X \Phi(|f_n-f|)d\mu\leq \frac{1}{2^i}.
 \end{equation*}
 Thus $\xymatrix{f_n\ar[r]^{d_\Phi}&f.}$ \\
 $\Leftarrow]$ Assume that $\xymatrix{f_n\ar[r]^{d_\Phi}&f.}$ Then, for all $i\in \mathbb{N},$ there exists $N_i,$ for all $n>N_i,$  $ \int_X \Phi(f_n-f)d\mu\leq \frac{1}{2^i}.$ 
 By the above proposition  $\|f_n-f\|_\Phi\leq k^{n_0(i)},$ where $n_0(i)=[\frac{\ln(\frac{1}{2^i})}{\ln(2)}]+1.$ We have $\|f_n-f\|_\Phi\leq k^{\frac{\ln(\frac{1}{2^i})}{\ln(2)}+1}\leq k^{-i\frac{\ln(2)}{\ln(2)}+1}\leq \frac{1}{k^{i-1}}.$
 Thus $\xymatrix{f_n\ar[r]^{\|.\|_\Phi}&f.}$
 \end{proof}
We conclude this section  by the following result:
 
 \begin{thm}
 $L_\Phi(X)$ is a quasi-Banach space.
 \end{thm}
 
 \begin{proof}
 Is obvious by proposition \ref{prop2}, as $(L_\Phi, d_\Phi)$ is complet space.
 \end{proof}

\section{\textbf{Linear form of $L_\Phi$ }}
In this last section we present the linear forms on $L_\Phi$ by discussing each case of the measure. \\

\textbf{Case 1:}  
Let $\mu(X)<+\infty$
\begin{lem}\cite{D} If $\Lambda \neq \emptyset, $ there exists a finite or countable sequence of sets $\{A_i \}\subset \mathcal{U}$ such that every $A \in \mathcal{U}$ differs from just one $\{A_i \}$ only by sets of measure zero. 
\end{lem}
In what follows we shall let $\{A_i \}$ be the family of sets  of $\Lambda$ of the preceding lemma, let $a_i=\mu(A_i)$ 
and let $\overline{f}_i$ be the value of the measurable function $f$ almost everywhere on $\{A_i \}.$ Then we have
\begin{thm}\label{thm4.2}
Any linear functional on $L_\Phi(X)$ is identically zero, if $\Lambda=\emptyset,$ or can be expressed in the form 
$U(f)=\sum_{i}u_i \overline{f}_i$ where $\|U\|= \sup_i |u_i|\Phi^{-1}(\frac{1}{a_i}).$\\
 If $u_i$ are given so that $\displaystyle|u_i|<k \frac{1}{\Phi^{-1}(\frac{1}{a_i})}$ for all $i,$ the function $U:f\mapsto \sum_i u_i \overline{f}_i $ is linear on $L_\Phi(X).$ 
\end{thm}

\begin{proof}
It is easily seen from Theorem 4.5, applied to $L_\Phi(X-\cup_i A_i),$ that $U(f)$ is independent of the values of $f$ on $X-\cup_i A_i ;$ hence no generality is lost if we assume $X=\cup_iA_i$ for simplicity.\\ Let $\chi_{A_i}$ be the characteristic function on $A_i,$ and take any $f\in L_\Phi(X),$ then $$\|\sum_{i=1}^n \overline{f}_i\xi_{A_i}-f\|_\Phi\rightarrow 0 \ \mbox{ as}\ n\rightarrow +\infty,$$ so by continuity, $$U(f)=\lim\limits_{n\rightarrow+\infty} U(\sum_{i=1}^{n}\overline{f}_i \chi_{A_i})=\lim\limits_{n\rightarrow+\infty}\sum_{i=1}^{n} U(\overline{f}_i \chi_{A_i})=\sum_{i=1}^{n}\overline{f}_i u_i,$$ where $u_i=U(\chi_{A_i}).$ Postponing the computation of $\|U\|$ for a moment.\\ Let us assume $u_i$ given, such that $\displaystyle|u_i|<k\frac{1}{\Phi^{-1}(\frac{1}{a_i})},$ and take $\|f\|_\Phi\leq 1.$
 Then  $$|\Phi(|\overline{f}_i|)a_i|\leq \|f\|_\Phi\leq 1,$$ so that $$|\overline{f}_i|=\Phi^{-1}(\Phi(|\overline{f}_i)a_i\times \frac{1}{a_i})\leq \Phi(\overline{f}_i)a_i\Phi^{-1}(\frac{1}{a_i})$$ since $\Phi^{-1}$ is convex, therefore
 \begin{eqnarray*}  |U(f)|\leq \sum |u_i\overline{f}_i|
 &\leq& k\sum_{i=1}^{+\infty}\frac{1}{\Phi^{-1}(\frac{1}{a_i})}\times \Phi(|\overline{f}_i|)a_i \Phi^{-1}(\frac{1}{a_i})\\
 &\leq& k\sum_{i=1}^{+\infty}\Phi(|\overline{f}_i|)a_i \\
 &\leq& k\int_{X}\Phi(|f|)d\mu\\ 
 &\leq& k,
 \end{eqnarray*}
 
  so $U$ is bounded hence it is continuous, and $\|U\|\leq k.$

It follows, if $k=\sup_i |u_i|\Phi^{-1}(\frac{1}{a_i}),$ that $\|U\|\geq k$ also for $$|U(\Phi^{-1}(\frac{1}{a_i}\chi_{A_i}))|=|u_i|\Phi^{-1}(\frac{1}{a_i})\leq \|U\|,$$ therefore, $\sup_i|u_i|\Phi^{-1}(\frac{1}{a_i})\leq \|U\|.$ This shows  that the series $\sum |u_i \overline{f}_i| $ converges. 
\end{proof}
\textbf{Case 2:} $\mu(X)=+\infty$\\
Let's first present the following Lemmma:
\begin{lem}\label{lemma1}
We have for $\Lambda\neq\emptyset$ if and only if $L_\Phi^* \neq0.$
\end{lem}
\begin{proof} B	y using the theorem 5.4 and 5.5 in \cite{Arabi}. 
\end{proof}
 For greater convenience we assume the $(X,\mathfrak{m},\mu)$ is $\sigma-$finite, (or $X\in B  $), that is, there exists  a sequence of measurable subsets $X_i\in \mathfrak{m} $ such that $X=\bigcup_i^{\infty}X_i$ and $\mu(X_i)<+\infty$ for all $i\in \mathbb{N}.$ We let $E_i,\ a_i$ and $\overline{f}_i$ have their previous meanings. 
\begin{thm}
If $(X,\mathfrak{m},\mu)$ is $\sigma-$finite, a functional $U$ on $L_\Phi(X)$ is linear if and only if: 
\begin{enumerate}
\item[a)] it is identically zero and $\Lambda=\emptyset,$ \\
or 
\item[b)] $\Lambda\neq\emptyset,$ and $U$ can be expressed in the form $U(f)=\sum_iu_i\overline{f}_i$ with $\|U\|=\sup_i|u_i|\Phi^{-1}(\frac{1}{a_i}).$
\end{enumerate}
\end{thm}
\begin{proof}
The sets $X_i$ which exists as $(X,\mathfrak{m},\mu)$ is $\sigma-$finite can obviously be taken disjoint; we let $U_j$ be a linear functional on $L_\Phi(X)$ defined by $U_j(f)=U(f.\chi_{X_j}).$ Then $U_j$ is linear on $L_\Phi(X_j),$ so by Theorem \ref{thm4.2} and Lemma \ref{lemma1} either $U_j$ is identically zero on $U_j(f)=\sum_iu_{ji}\overline{f}_{ji}.$ \\
Now $\lim\limits_{n\rightarrow +\infty}\|\sum_{i=1}^{n}f\chi_{X_i}-f\|_\Phi=0.$ So $U(f)=\lim\limits_{n\rightarrow\infty}U(f\chi_{X_j})=\sum_jU_j(f).$ Moreover there is an $f_0\in L_\Phi(X)$ such that $|f_0(y)|=|f(y)|$ for $y\in X$ and $U_j(f_0)=|U_j(f)|$ for all $j;$ hence
 $$|U_j(f)|\leq \sum_j|U_j(f)|=\sum_jU_j(f_0)=U(f_0)\leq \|U\|.\|f\|_\Phi,$$so $\|U_j\|\leq \|U\|$ and the series $\sum_jU_j(f)$ converges absolutely for each $f\in L_\Phi.$
  Then $$|u_{ji}|\Phi^{-1}(\frac{1}{a_{ji}})\leq \|U_j\|\leq \|U\|$$ and $U(f)=\sum_j\sum_iu_{ji}\overline{f}_{ji}$ unless all  $U_j$ are identically zero, that is, unless $\mathcal{U}=\emptyset.$ Rearranging, we get $U(f)=\sum u_if_i.$ \\
  The other conclusions follow as in Theorem \ref{thm4.2}. 
\end{proof}
\textbf{Case 3:}
There remains the case in which $(X,\mathfrak{m},\mu)$ is not  $\sigma-$finite. As we did in the proof of Lemma 7 in \cite{D}, we can define a well-ordered set $\{A_\gamma\}$ of sets of $\Lambda$ disjoint up to sets of measure of measure zero, equal to some $A_\gamma;$ however we have no assurance that the sequence will be countable. As before we let $a_\gamma=\mu(A_\gamma)$ and $\overline{f}_\gamma$ be the value of $f$ almost everywhere on $A_\gamma.$\\
Now, if $f\in L_\Phi(X)$ if and only if $\Phi(f)\in L^1(X).$ By Lemma 8 in \cite{D} $$\{x\in X,f(x)\neq 0\}=\{x\in X;\Phi(f)(x)\neq 0\}\in B.$$
 By this, it is possible to put each function $f$ of $L_\Phi(X)$ into at least one class $D_E$ such that $f(x)=0$ if $x\in X\backslash E$ and $E\in B;$ then $D_E$ is equivalent to $L_\Phi(E),$ and if we set 
 $D_E(f)=U(f)$ for $f\in D_E,$ $U_E$ is linear on $L_\Phi(E)$  and therefore can be expressed as before where the $A_i$ of theorem above will be those $E_\gamma$ which except for sets of measure zero, lie in $E.$ But $f=0$ on $X\backslash E;$ so we can write $U(f)=\sum_{\gamma}U_\gamma \overline{f}_\gamma $ if it is not identically zero, with the convention that the sum of any number of terms in which $\overline{f}_\alpha $  or $u_\alpha$ is zero shall be zero.\\
  Since this can be done for each $f\in L_\Phi(X),$ we get our final result, including the previous theorems as special cases.
 \begin{thm}
 A functional $U$ on $L_\Phi(X)$ is linear if and only if: 
 \begin{enumerate}
 \item[a)]  $\Lambda=\emptyset,$ and $U$ is identically zero \\
 or 
 \item[b)] $\Lambda\neq\emptyset,$ and $U$ can be expressed in the form $U(f)=\sum_\gamma u_i\overline{f}_\gamma,$ and  $\|U\|=\sup_\gamma|u_\gamma|\Phi^{-1}(\frac{1}{a_\gamma}).$
 \end{enumerate} 
 \end{thm}

\begin{rem}
$L_\Phi$ is a quasi-normed linear space , by using Theorem 3.4 in \cite{Gobardhan}  then the Dual space of $L_\Phi$ is a complete normed linear space. \end{rem}

\section{\textbf{The $L_\Phi$ Spaces with Dual Space 0}}

In this part we  speak about some topological vector spaces that are not in general locally convex. For example the function $d(f,g) =\int_{0}^{1} |f(x)- g(x)|^{1/2}dx$ is a metric on $L^{1/2}[0,1].$ The topology
$L^{1/2}[0,1]$	 obtains from this metric is not locally convex.
Theorem 3.2 in \cite{K} proved that $L^p(\mu)$  for $0<p<1$ is locally convex space where the measure $\mu$ assumes finitely many values.\\
 For a measure space $(X,\mathcal{M},\mu)$ and $\Phi$ is an $N^*-$function, is $L_\Phi$ ever locally convex?

\begin{thm}
$L_\Phi(X)$ is locally convex if and only if the measure $\mu$ assume finitely many values. 
\end{thm}
\begin{proof}
If $\mu$ takes finitely many values, then $X$ is the disjoint union of finitely many atoms
 $B_1,...,B_m.$ A measurable function is constant almost everywhere on each atom, so $L_\Phi(X)$ is topologically  just Euclidean space (of dimension equal to the number of atoms of finite measure) which is locally convex.\\
Now assume $\mu$ takes infinitely many values.\\
Since $\mu$ has infinitely many values, there is a sequence of subsets $Y_i\subset X$ such that $$0<\mu(Y_1)<\mu(Y_2)<...$$
 From the sets $Y_i,$ we can construct recursively a sequences of disjoint sets $A_i$ such that $\mu(A_i)>0.$\\
 Fix $\displaystyle \varepsilon>0.$ Let $f_k=\Phi(\displaystyle\frac{\varepsilon}{\mu(A_k)})\chi_k,$ so 
 \begin{equation*}
 \int_{X}\Phi(|f_k|)d\mu=\int_{X}\frac{\varepsilon}{\mu(A_k)}\chi_{A_k} d\mu=\frac{\varepsilon}{\mu(A_k)}\times \mu(A_k)=\varepsilon.
 \end{equation*}
If $L_\Phi(X)$ is locally convex, then any open set around $0$ contains a convex open set around $0,$ which in turn contains some
$\varepsilon-$ball (and thus every $f_k$). We will show an average of enough $f_k$'s is arbitrarily far from $0$
in the metric on $L_\Phi(X),$ and that will contradict local convexity.

Let  $h_n=\displaystyle \frac{1}{n}\sum_{k=1}^{n}f_k.$  Since $f_k$'s are supported on disjoint sets, $$\displaystyle \int_{X}\Phi(|h_n|)d\mu=\sum_{k=1}^{n}\int_{A_k}\Phi(|h_n|)d\mu=\sum_{k=1}^{n}\varepsilon \Phi(\frac{1}{n})=n\Phi(\frac{1}{n})\varepsilon.$$  Since $\displaystyle \lim\limits_{x\rightarrow 0^+}\frac{\Phi(x)}{x}=+\infty,$ thus $\displaystyle\lim\limits_{n\rightarrow +\infty}\varepsilon n \Phi(\frac{1}{n})=\displaystyle\varepsilon \lim\limits_{n\rightarrow +\infty}\frac{\Phi(\frac{1}{n})}{\frac{1}{n}}=+\infty, $ thus, $L_\Phi(X)$ is not locally convex.   
 
\end{proof}

\begin{lem}\label{lemma1}
If $(X,\mathcal{M},\mu)$ is a non-atomic measure space, then for any non negative $F\in \mathcal{L}^1(\mu,\mathbb{R}),$  $Fd\mu$  is a finite non-atomic measure. 
\end{lem}

\begin{lem}\label{lemma2}
If $(X,\mathcal{M},\lambda)$ is any finite  non-atomic measure space, then for all $t$ in  $[0,\nu(X)]$ there is  some $A\in \mathcal{M}$ such that   $\nu(A)=t.$  
\end{lem}

\begin{thm}
If $(X,\mathcal{M},\mu)$ is a non-atomic measure space, then $L_\Phi(X)^*=0.$
\end{thm}
\begin{proof}
We argue by contradiction. Assume there is  $\varphi \in L_\Phi(X)^*$ with $\varphi \neq 0.$ Then $\varphi$ has image $\mathbb{R}$ (a non zero linear map to a one-dimensional space is surjective ),so there is some $f\in L_\Phi(X) $ such that $|\varphi(f)|\geq 1.$\\
When $(X,\mathcal{M},\mu)$ is non-atomic and $\Phi(f)\in L^1(X),$ Lemma \ref{lemma1} tells us $\Phi(|f|)d\mu$ is a finite non-atomic  measure on $X,$ and Lemma \ref{lemma2} then tells us there is an $A\in \mathcal{M}$ such that $\displaystyle \int_{A}\Phi(|f|)d\mu=\frac{1}{3}\int_X\Phi(|f|)d\mu.$ Then set $g_1=f \chi_A$ and $g_2=f \chi_{X-A}$ so $f=g_1+g_2$ and $\Phi(|f|)=\Phi(|g_1|)+\Phi(|g_2|).$ So 
\begin{equation*}
\displaystyle  \int_{X}\Phi(|g_1|)d\mu= \int_{A}\Phi(|f|)d\mu=\frac{1}{3} \int_{X}\Phi(|f|)d\mu.
\end{equation*}
Hence 
\begin{equation*}
\displaystyle  \int_{X}\Phi(|g_2|)d\mu= \frac{1}{3} \int_{X}\Phi(|f|)d\mu.
\end{equation*}
Since $\displaystyle|\varphi(|f|)|\geq 1, \ |\varphi(|g_i|)|\geq \displaystyle\frac{1}{2}  $ for some $i.$ Let  $f_1=2g_i,$ so $|\varphi(f_1)|\geq 1$ and
\begin{eqnarray*}
\displaystyle \int_{X}\Phi(|f_1|)d\mu=\int_{X}\Phi(|2g_1|)d\mu
&\leq&2\int_{X}\Phi(|g_1|)d\mu\\
&\leq&\frac{2}{3}\int_{X}\Phi(|f|)d\mu.
\end{eqnarray*}
 Iterate this to get a sequence $\{f_n\}$ in $L_\Phi (X)$ such that $|\varphi(f_n)|\geq 1$ and $d(f_n,0)=\displaystyle \int_X\Phi(|f_n|)d\mu\leq \displaystyle \big(\frac{2}{3}\big)^n\int_X\Phi(|f|)d\mu .$ Then $\lim\limits_{n\rightarrow +\infty}d(f_n,0)=0,$ a contradiction of continuity of $\varphi.$
\end{proof}

\begin{thm}
If the measure $\mu$  contains an atom with finite measure, then $L_\Phi^*(X)\neq 0.$
\end{thm}
\begin{proof}
Let $B$ be an atom with finite measure. Any measurable function $f:X\rightarrow \mathbb{R}$ is constant almost everywhere on $B.$ Call the almost everywhere common value $\varphi(f).$ The reader can check $\varphi$ is a non-zero  continuous linear functional on $L_\Phi(X).$
\end{proof}

 \begin{rem}
   If $L^*_\Phi\neq 0,$  how to express the elements of $L^*_\Phi \backslash \{0\}?$\\
  If $A\in L^*_\Phi$ and any $f\in L_\Phi.$ By the density of $L^1,$ there exists sequence $(f_n)$ in $L^1$ such that $d_\Phi(f,f_n)\rightarrow 0,$ as $n\rightarrow +\infty.$ We have,  
    $$A(f)=A(\lim\limits_{n\rightarrow+\infty}f_n)=\lim\limits_{n\rightarrow+\infty}A(f_n),$$ 
by Lemma 2 in \cite{D}, we get  $A(f)=\displaystyle\lim\limits_{n\rightarrow+\infty}\int_Xf_nud\mu,$  where $u$ is a bounded function in $X.$ \\\\ 
   
   \end{rem}

\hspace{-0.3cm}\textbf{\small{Author contributions}} \small{The study was carried out in collaboration of all authors. All authors read and
approved the final manuscript.}\\

\hspace{-0.3cm}\textbf{\small{Funding}} \small{Not Applicable.}\\

\hspace{-0.3cm}\textbf{\large{Compliance with ethical standards}}\\

\hspace{-0.3cm}{\textbf{\small{Conflicts of interest}}  \small{The authors declare that they have no conflict of interest.}\\

\hspace{-0.3cm}{\textbf{\small{Ethical approval}}\small{ This article does not contain any studies with human participants or animals performed
by any of the authors.}

\bibliographystyle{plain}

\end{document}